\documentclass[12pt]{article}
\usepackage{amssymb}
\usepackage{mathrsfs}
\usepackage[hypertex]{hyperref}
\usepackage{amsfonts}
\usepackage{graphicx}
\usepackage{latexsym,amsmath,amssymb,amsfonts,amsthm}

\newcommand{\bcen}{\begin{center}}
\newcommand{\ecen}{\end{center}}
\newtheorem{theorem}{Theorem}[section]

\newtheorem{corollary}[theorem]{Corollary}

\newtheorem{remark}[theorem]{Remark}

\newtheorem{defn}[theorem]{Definition}

\begin{document}
\setcounter{page}{1}
\title{Eigenvalue comparison theorems for the Witten-Laplacian and the weighted $p$-Laplacian on complete manifolds with a modified Ricci
curvature bounded from below}
\author{Ruifeng Chen$^{1}$,~~ Qiyue Ma$^{1}$, ~~Jing Mao$^{1,2,\ast}$}

\date{}
\protect \footnotetext{\!\!\!\!\!\!\!\!\!\!\!\!{~MSC 2020:
58C40, 58J50, 35P15.}
{$\ast$ Corresponding author} }
\maketitle ~~~\\[-15mm]

\begin{center}
{\footnotesize $^{1}$Faculty of Mathematics and Statistics,\\
 Key Laboratory of Applied
Mathematics of Hubei Province, \\
Hubei University, Wuhan 430062, China\\
$^{2}$Key Laboratory of Intelligent Sensing System and Security
(Hubei
University), Ministry of Education\\
Email: jiner120@163.com
 }
\end{center}


\begin{abstract}
In this paper, for complete manifolds with a modified Ricci
curvature bounded from below, we can successfully establish
Cheng-type eigenvalue comparison theorems for the first Dirichlet
eigenvalues of the Witten-Laplacian and the weighted $p$-Laplacian
on geodesic balls of these manifolds.
 \end{abstract}


\section{Introduction}
\renewcommand{\thesection}{\arabic{section}}
\renewcommand{\theequation}{\thesection.\arabic{equation}}
\setcounter{equation}{0}

\subsection{Motivation}

The study of comparison theorems between well-investigated geometric
objects and less-investigated ones is a hot topic in Geometry.
Sometimes, comparison theorems play a fundamental role in seeking
geometric and topological properties of manifolds with curvature
constraints.

For a given complete Riemannian $n$-manifold $M^{n}$, $n\geq2$, if
the Ricci curvature of $M^{n}$ satisfies
$\mathrm{Ric}(M^n)\geq(n-1)K$ for some constant $K$,  S. Y. Cheng
\cite{CSY1} showed that for arbitrary $q\in M^{n}$, the comparison
\begin{eqnarray} \label{ECT-1}
\lambda_{1}(B(q,r_0))\leq\lambda_{1}\left(\mathcal{B}_{n}(K,r_0)\right)
\end{eqnarray}
holds, where the equality in (\ref{ECT-1}) holds if and only if
$B(q,r_0)$ is isometric to $\mathcal{B}_{n}(K,r_0)$. Here $B(q,r_0)$
denotes the geodesic ball on $M^{n}$, with center $q$ and radius
$r_0$, while $\mathcal{B}_{n}(K,r_0)$ stands for a geodesic ball of
radius $r_0$ in the space $n$-form $\mathbb{M}^{n}(K)$ with constant
sectional curvature $K$. Besides, $\lambda_{1}(\cdot)$ denotes the
first Dirichlet eigenvalue of the Laplacian on given geometric
objects. By the way, one does not need to give the information about
the center of the geodesic ball $\mathcal{B}_{n}(K,r_0)$ since
$\mathbb{M}^{n}(K)$ is two-points homogeneous. If the sectional
curvature of $M^{n}$ satisfies $\mathrm{Sec}(M^n)\leq K$, S. Y.
Cheng \cite{CSY2} showed that for $q\in M^{n}$, the comparison
\begin{eqnarray} \label{ECT-2}
\lambda_{1}(B(q,r_0))\geq\lambda_{1}\left(\mathcal{B}_{n}(K,r_0)\right)
\end{eqnarray}
holds, where $B(q,r_0)$ should be within the cut locus of the point
$q$, and the equality in (\ref{ECT-2}) holds if and only if
$B(q,r_0)$ is isometric to $\mathcal{B}_{n}(K,r_0)$. These two
comparisons, together  with related rigidity conclusions, are called
``\emph{Cheng's eigenvalue comparison theorems}" in some
literatures. In our viewpoints, the key steps (in \cite{CSY1, CSY2})
used for the derivation of Cheng's eigenvalue comparison theorems
(\ref{ECT-1})-(\ref{ECT-2}) are actually the followings:
\begin{itemize}
\item The suitable construction of trial functions, which highly
depends on eigenfunctions of the first Dirichlet eigenvalue of the
Laplacian on $\mathcal{B}_{n}(K,r_0)$.

\item The usage of Bishop's volume comparison theorems (see e.g. \cite[Chapter III]{IC}).
\end{itemize}
Eigenvalue comparisons (\ref{ECT-1})-(\ref{ECT-2}) not only give the
rigidity characterization for the equality case but also supply the
accurate (upper or lower) bound for the eigenvalue
$\lambda_{1}(B(q,r_0))$. This is because, once the accurate values
of the curvature $K$ and the radius $r_0$ were fixed,
$\lambda=\lambda_{1}\left(\mathcal{B}_{n}(K,r_0)\right)$ can be
calculated accurately and is actually determined by the system
\begin{eqnarray} \label{firste-1}
\left\{
\begin{array}{lll}
\frac{d^{2}\varphi(t)}{dt^{2}}+(n-1)\frac{S'_{K}(t)}{S_{K}(t)}\frac{d\varphi(t)}{dt}+\lambda\cdot\varphi(t)=0 \qquad &\mathrm{in}~\left(0,r_0\right),  \\[1mm]
\varphi'(0)=0,~\varphi\left(r_0\right)=0, \\[0.5mm]
 \varphi|_{(0,r_0)}>0,
\end{array} \right.
\end{eqnarray}
with
\begin{eqnarray*}
S_{K}(t)=\left\{
\begin{array}{llll}
\frac{\sin\sqrt{K}t}{\sqrt{K}}, & \quad K>0,\\
 t, & \quad K=0, \\
\frac{\sinh\sqrt{-K}t}{\sqrt{-K}}, & \quad K<0.
\end{array}
\right.
\end{eqnarray*}
Specially, when $K=0$, $\mathbb{M}^{n}(K)$ degenerates into the
Euclidean $n$-space $\mathbb{R}^n$, solving (\ref{firste-1})
directly yields that
$$\lambda_{1}\left(\mathcal{B}_{n}(0,r_0)\right)=\left(\frac{j_{\frac{n}{2}-1,1}}{r_0}\right)^{2},$$
where $j_{\frac{n}{2}-1,1}$ denotes the first zero of the
$\left(\frac{n}{2}-1\right)$-st Bessel function.

Cheng's eigenvalue comparisons (\ref{ECT-1})-(\ref{ECT-2}) inspire
some subsequent works -- see e.g. \cite{DDMZ, FMI,  JM4, JM-1, SAG,
YWMD}, where readers would find that eigenvalue comparisons obtained
therein highly depend on curvature assumptions. This fact motivates
us to consider:

\begin{itemize}
\item (\textbf{Question 1}). Is it possible to get some other extensions to Cheng's
eigenvalue comparisons (\ref{ECT-1})-(\ref{ECT-2}) under different
and feasible curvature assumptions?
\end{itemize}

This paper is organized as follows. Some preliminaries (including
the precise definitions of two elliptic differential operators,
basic facts of the Dirichlet eigenvalue problems of those two
operators, a modified Ricci curvature, the concept of spherically
symmetric manifolds, etc) will be shown in Subsection
\ref{subsect2}. The main conclusions of this paper will be stated
clearly in Subsection \ref{subsect3} (i.e. An affirmative answer to
\textbf{Question 1}). A generalization of Bishop's classical volume
comparison will be proven in Section \ref{sect3}. The proofs of our
Cheng-type eigenvalue comparisons for the Witten-Laplacian and the
weighted $p$-Laplacian (i.e. Theorems \ref{maintheorem1} and
\ref{maintheorem2}) will be given in Section \ref{sect4}, Section
\ref{sect5} respectively.

\subsection{Necessary preliminaries} \label{subsect2}

Throughout this paper, let $M^{n}$ be an $n$-dimensional ($n\geq2$)
complete Riemannian manifold with the metric
$\langle\cdot,\cdot\rangle$, and let $w$ be a given smooth strictly
\emph{positive} function on $M^{n}$. As explained by A. G. Setti
\cite[Section 1]{SAG} that, no matter from the aspect of the
weighted setting but also from the aspect of  right-invariant (sub)
Laplacians on a Lie group with left Haar measure, weighted Laplacian
arise naturally and it should be meaningful to study them. On
$(M^{n}, \langle\cdot,\cdot\rangle)$, one can define as in
\cite{SAG} the following elliptic differential operator (which is a
weighted version of the Laplacian)
\begin{eqnarray*}
\mathcal{L}:= -\Delta - \langle \nabla(\log w), \nabla \cdot
\rangle,
\end{eqnarray*}
and a modified Ricci curvature
\begin{eqnarray*}
R_{w}:=\mathrm{Ric}-w^{-1}\mathrm{Hess}w,
\end{eqnarray*}
where $\Delta$, $\nabla$ are the Laplace and the gradient operators
on $M^{n}$ respectively, $\mathrm{Ric}$ stands for the Ricci
curvature tensor on $M^{n}$, and $\mathrm{Hess}$ denotes the Hessian
operator. For a bounded domain $\Omega$ on $M^{n}$, we can consider
the following Dirichlet eigenvalue problem
\begin{eqnarray} \label{eigenp-1}
\left\{
\begin{array}{ll}
\mathcal{L}u=\lambda u  \qquad & \mathrm{in}~\Omega\subset M^n,\\
u=0\qquad & \mathrm{on}~\partial\Omega.
\end{array}
\right.
\end{eqnarray}
A nonlinear version of the operator $\mathcal{L}$ can be
well-defined as follows
\begin{eqnarray*}
\mathcal{L}_{p}: = -\frac{1}{w} \mathrm{div}\left(w |\nabla
\cdot|^{p-2} \nabla\cdot \right),
\end{eqnarray*}
with $1<p<\infty$, where $\mathrm{div}$ stands for the divergence
operator on $M^n$. For a function $u$ defined on $M^{n}$, in a local
coordinates chart  $\{x^1, x^2, \ldots, x^n\}$ of $M^{n}$,   one
easily has
\begin{eqnarray*}
\mathcal{L}_{p}u=-\frac{1}{w\sqrt{\det
[g_{ij}]}}\sum_{i,j=1}^{n}\frac{\partial}{\partial x^{i}}
\left(w\sqrt{\det [g_{ij}]}|\nabla u|^{p-2}g^{ij}\frac{\partial
u}{\partial x^{j}}\right),
\end{eqnarray*}
where $\det [g_{ij}]$ is the determinant of the metric matrix
$[g_{ij}]$, and $(g^{ij})=[g_{ij}]^{-1}$ denotes the inverse of this
 metric matrix. For a bounded domain $\Omega$ on $M^{n}$, the following Dirichlet eigenvalue problem
\begin{eqnarray} \label{eigenp-wpl}
\left\{
\begin{array}{ll}
\mathcal{L}_{p}u=\lambda|u|^{p-2}u  \qquad & \mathrm{in}~\Omega\subset M^n,\\
u=0\qquad & \mathrm{on}~\partial\Omega
\end{array}
\right.
\end{eqnarray}
can be considered as well.

\begin{remark}  \label{remark1-1}
\rm{ As pointed out in \cite[Remark 1.22]{JM-1}, for the smooth
strictly positive weight function $w$ defined on $(M^{n},
\langle\cdot,\cdot\rangle)$, there must exist a smooth real-valued
function $\phi$ (which is called the \emph{potential function}) such
that $e^{-\phi}=w$. Then the Witten-Laplacian (also called weighted
Laplacian, drifting Laplacian, or $\phi$-Laplacian) and the
$N$-dimensional Bakry-\'{E}mery Ricci curvature tensor can be
well-defined as follows
\begin{eqnarray*}
\Delta_{\phi}:=\Delta-\langle\nabla\cdot,\nabla\phi\rangle=e^{\phi}\mathrm{div}(e^{-\phi}\nabla\cdot)
\end{eqnarray*}
and
\begin{eqnarray*}
\mathrm{Ric}^{N}_{\phi}:=\mathrm{Ric}+\mathrm{Hess}\phi-\frac{d\phi\otimes
d\phi}{N-n}.
\end{eqnarray*}
For the $N$-dimensional Bakry-\'{E}mery Ricci curvature
$\mathrm{Ric}^{N}_{\phi}$, one knows $N>n$ or $N=n$ if $\phi$ is a
constant function. Specially, when $N=\infty$, one can define the
so-called $\infty$-dimensional Bakry-\'{E}mery Ricci curvature
(simply, Bakry-\'{E}mery Ricci curvature or weighted Ricci
curvature) as follows
 \begin{eqnarray*}
 \mathrm{Ric}_{\phi}=\mathrm{Ric}+\mathrm{Hess}\phi.
 \end{eqnarray*}
These notions can be found in \cite{BE} and their relationship with
diffusion processes has also been investigated therein. It is not
hard to see
 \begin{eqnarray*}
\mathcal{L}=-\Delta_{\phi}.
 \end{eqnarray*}
Besides, in \cite[Remark 1.22]{JM-1}, it has been shown that
\begin{eqnarray*}
 \mathrm{Ric}_{w}=\mathrm{Ric}-e^{\phi}\mathrm{Hess}(e^{-\phi})=\mathrm{Ric}+\mathrm{Hess}\phi-d\phi\otimes
d\phi=
 \mathrm{Ric}_{\phi}-d\phi\otimes
d\phi,
 \end{eqnarray*}
which implies that the only difference between $\mathrm{Ric}_{w}$
and the $N$-dimensional Bakry-\'{E}mery Ricci curvature tensor
$\mathrm{Ric}_{\phi}^{N}$ is the coefficient $\frac{1}{N-n}$ in
front of the term $-d\phi\otimes d\phi$, and it is not an essential
difference. By the way, it is not hard to check that the nonlinear
operator $\mathcal{L}_{p}$ satisfies
 \begin{eqnarray*}
\mathcal{L}_{p}=-\Delta_{p,\phi}
 \end{eqnarray*}
provided that $w$ is rewritten as $w=e^{-\phi}$, where
$\Delta_{p,\phi}$ is the weighted $p$-Laplacian well-defined as in
\cite[Section 2]{JM-1}. Clearly, if the potential function
$\phi=const.$ is a constant function (i.e. the weight function $w$
is constant), then the Witten-Laplacian $\Delta_{\phi}$ (i.e.
$-\mathcal{L}$), the weighted $p$-Laplacian $\Delta_{p,\phi}$ (i.e.
$-\mathcal{L}_{p}$) degenerate into the Laplacian $\Delta$ and the
$p$-Laplacian
$\Delta_{p}=\mathrm{div}(|\nabla\cdot|^{p-2}\nabla\cdot)$,
respectively. }
\end{remark}

Since $\mathcal{L}=-\Delta_{\phi}$ explained in Remark
\ref{remark1-1}, one has from \cite[Section 1]{JM-1} (see also
\cite[pp. 1247-1248]{CM}) that:

\begin{itemize}
\item The eigenvalue problem (\ref{eigenp-1}) only has discrete
spectrum, and all the elements (i.e. eigenvalues) in this discrete
spectrum can be listed non-decreasingly as follows
 \begin{eqnarray} \label{sequen-1}
 0<\lambda_{1,w}(\Omega)<\lambda_{2,w}(\Omega)\leq\lambda_{3,w}(\Omega)\leq\cdots\uparrow+\infty.
 \end{eqnarray}
Each eigenvalue\footnote{~For convenience and without any confusion,
except specification we wish to write $\lambda_{i,w}(\Omega)$ as
$\lambda_{i,w}$ directly. This abbreviation convention would also be
used for eigenvalues of other types discussed in the sequel.}
$\lambda_{i,w}$, $i=1,2,\ldots$, in the sequence (\ref{sequen-1})
was repeated according to its multiplicity (which is finite and
equals the dimension of the eigenspace of $\lambda_{i,w}$).

 \item By the variational principle, it is not hard to see that the
 first Dirichlet eigenvalue $\lambda_{1,w}(\Omega)$ can be
 characterized as follows
  \begin{eqnarray} \label{chr-1}
 \lambda_{1,w}(\Omega)=\inf\left\{\frac{\int_{\Omega}|\nabla u|^{2}wdv}{\int_{\Omega}u^{2}wdv}=
 \frac{\int_{\Omega}|\nabla u|^{2}d\mu}{\int_{\Omega}u^{2}d\mu}\Bigg{|}u\in
 W^{1,2}_{0,w}(\Omega),u\neq0\right\},
  \end{eqnarray}
  where $dv$ is the volume element of the Riemannian manifold $M^n$,
  and $d\mu:=wdv$. Using notations similar as in \cite{CM}, here $W^{1,2}_{0,w}(\Omega)$
denotes a Sobolev space, which is the completion of the set of
smooth functions (with compact support) $C^{\infty}_{0}(\Omega)$
under the following Sobolev norm
\begin{eqnarray*}
\|u\|^{w}_{1,2}:=\left(\int_{\Omega}u^{2}wdv+\int_{\Omega}|\nabla
u|^{2}wdv\right)^{1/2}.
\end{eqnarray*}
Moreover, the $k$-th Dirichlet eigenvalue $\lambda_{k,w}(\Omega)$
can be characterized as follows
 \begin{eqnarray*}
 \lambda_{k,w}(\Omega)=\inf\left\{\frac{\int_{\Omega}|\nabla u|^{2}d\mu}{\int_{\Omega}u^{2}d\mu}
 \Bigg{|}u\in W^{1,2}_{0,w}(\Omega),u\neq0,\int_{\Omega}uu_{i}d\mu=0\right\},
 \end{eqnarray*}
where $u_{i}$, $i=1,2,\cdots,k-1$, denotes an eigenfunction of
$\lambda_{i,w}(\Omega)$.
\end{itemize}
In our previous work \cite[Introduction]{CM1}, we have given a
detailed explanation on why is interesting and important to study
spectral geometric problems related to the Witten-Laplacian. We
already have some interesting works about spectral isoperimetric
inequalities, spectral estimates and geometric functional
inequalities related to the Witten-Laplacian -- see, e.g., \cite{CM,
CM1, DMWW, JM1, JM2, YWMD}.

Since $\mathcal{L}_{p}=-\Delta_{p,\phi}$ explained in Remark
\ref{remark1-1}, one directly has from \cite[Section 2]{JM-1} that:
\begin{itemize}
\item (\ref{eigenp-wpl}) has a positive weak solution, which is unique modulo
the scaling, in the space $W^{1,p}_{0,w}(\Omega)$, the completion of
the set $C^{\infty}_{0}(\Omega)$ of smooth functions compactly
supported on $\Omega$ under the Sobolev norm
\begin{eqnarray*}
\|u\|^{w}_{1,p}:=\left(\int_{\Omega}|u|^{p}wdv+\int_{\Omega}|\nabla
u|^{p}wdv\right)^{1/p}.
\end{eqnarray*}

\item By applying the
Ljusternik-Schnirelman principle, there exists a nondecreasing
sequence of nonnegative eigenvalues $\{\lambda_{i,p}^{w}(\Omega)\}$
tending to $\infty$ as $i\rightarrow\infty$.

\item The first Dirichlet eigenvalue $\lambda_{1,p}^{w}(\Omega)$  is simple, isolated, and eigenfunctions
associated with $\lambda_{1,p}^{w}(\Omega)$ do not change sign.
Besides, $\lambda_{1,p}^{w}(\Omega)$ can be characterized by
\begin{eqnarray} \label{chr-2}
\lambda_{1,p}^{w}(\Omega)=\inf\left\{\frac{\int_{\Omega}|\nabla
u|^{p}wdv}{\int_{\Omega}| u|^{p}wdv}=\frac{\int_{\Omega}|\nabla
u|^{p}d\mu}{\int_{\Omega}| u|^{p}d\mu}\Big{|}u\neq0,~u\in
W^{1,p}_{0,w}(\Omega)\right\}.
\end{eqnarray}

\item The set of eigenvalues is closed. The eigenvalue
$\lambda_{2,p}^{w}(\Omega)$ is the second eigenvalue, i.e.
$\lambda_{2,p}^{w}(\Omega)=\inf\{\lambda|\lambda~{\rm{is~an~eigenvalue~of~(\ref{eigenp-wpl})~and}}~\lambda>\lambda_{1,p}^{w}(\Omega)\}$.

\end{itemize}
In \cite[Section 3]{DMWX}, we have given several lower bounds for
the first eigenvalue of weighted $p$-Laplacian on submanifolds with
locally bounded weighted mean curvature.

If the curvature assumption that the Ricci curvature is bounded from
below by some constant was replaced by the one that the modified
Ricci curvature $R_{w}$ is bounded from below by some constant, A.
G. Setti \cite[Theorem 4.2]{SAG} generalized the eigenvalue
comparison (\ref{ECT-1}) in some way, but lost the related rigidity
characterization. A little bit more than 10 years ago, P. Freitas,
J. Mao, I. Salavessa \cite[Definition 2.1]{FMI} accurately gave the
notion of \emph{spherical symmetrization} as follows:
\begin{defn}
A domain $\Omega=\exp_q([0,l)\times{S}_q^{n-1}) \subset
M^{n}\backslash \mathrm{Cut}(q)$, with $l<inj(q)$,  is said to be
spherically symmetric with respect to a point $q\in \Omega$, if
 and only if
the matrix $\mathbb{A}(t,\xi)$ satisfies $\mathbb{A}(t,\xi)=f(t)I$,
for a function $f\in{C^{2}([0,l))}$,  with   $f(0)=0$, $f'(0)=1$,
and  $f|(0,l)>0$.
\end{defn}
\noindent Here\footnote{~In this paper, since we also consider the
weighted $p$-Laplacian case, in order to avoid any potential
confusion, we denote  by $q$ the chosen point on $M^n$.} $\exp_q$
denotes the exponential mapping at the point $q\in
 M^{n}$, $\mathrm{Cut}(q)$ and $inj(q)$ stand for the cut-locus of $q$ and the injectivity radius at
 $q$ respectively, ${S}_q^{n-1}$ is the unit sphere in the tangent
 space $T_{q}M^{n}$ at $q$, $I$ denotes the identity matrix, and $\mathbb{A}(t,\xi)$ is the path of linear
 transformations defined similarly as in \cite[p. 703]{FMI} (see also the definition (\ref{PLT}) in Section \ref{sect3} below).
 Therefore, in this case the Riemannian metric of $M^{n}$ can be
 expressed on $\Omega$ by
 \begin{eqnarray*}
ds^{2}=dt^{2}+f^{2}(t)|d\xi|^2, \qquad \forall\xi\in
S^{n-1}_{q},~0<t<l,
 \end{eqnarray*}
with $|d\xi|^2$ the round metric on the unit Euclidean
$(n-1)$-sphere $\mathbb{S}^{n-1}$. If the given complete
$n$-manifold has\footnote{~This concept of radial (Ricci or
sectional) bounded can be seen in \cite[Definition 2.2 and
Definition 2.3]{FMI}.} a \emph{radial} Ricci (or
 sectional) curvature lower (or upper) bound with respect to a point
 $q\in M^{n}$, and this bound is a continuous function of
 the Riemannian distance parameter $r(\cdot):=d(q,\cdot)$ within the set
 $M^{n}\setminus\left({q}\cup\mathrm{Cut}(q)\right)$, they can
 set up two Bishop-type volume comparisons (see \cite[Corollary 3.5 and Theorem
 4.2]{FMI}). For any $x\in
 M^{n}\setminus\left({q}\cup\mathrm{Cut}(q)\right)$, there exists a
 unit vector $\xi\in S_{q}^{n-1}$ such that
 $x=\gamma_{\xi}(t_0)=\exp_{q}(t_{0}\xi)$, $0<t_{0}<d_{\xi}(q)$,
 where
 \begin{eqnarray*}
~~d_{\xi}(q):=\sup\{t>0|\gamma_{\xi}(s)=\exp_{q}(s\xi)~\mathrm{is~the~unique~minimal~geodesic~joining}~q~\mathrm{and}~\gamma_{\xi}(t)\}.
 \end{eqnarray*}
As explained in \cite[Section 2]{FMI}, one knows  for any $x\in
 M^{n}\setminus\left({q}\cup\mathrm{Cut}(q)\right)$ that
\begin{eqnarray*}
\nu_{x}=\nabla r(x)=\frac{d}{dt}\Big{|}_{x}
\end{eqnarray*}
is the radial unit tangent vector at $x=\gamma_{\xi}(t)$. By
applying two Bishop-type volume comparisons obtained in
\cite[Corollary 3.5 and Theorem
 4.2]{FMI} as part of the main tools, for an $n$-dimensional ($n\geq2$) complete Riemannian manifold $M^{n}$, P. Freitas,
J. Mao, I. Salavessa \cite[Theorems 3.6 and 4.4]{FMI} obtained two
eigenvalue comparisons as follows:
\begin{itemize}
\item If the radial Ricci curvature of $M^{n}$ has a lower bound
$(n-1)\kappa(t)$ with respect to a point $q\in M^{n}$, i.e.
\begin{eqnarray*}
\mathrm{Ric}(\nu_{x},\nu_{x})\geq(n-1)\kappa(t)
\end{eqnarray*}
for any $x=\gamma_{\xi}(t)=\exp_{q}(t\xi)$, $0<t<d_{\xi}(q)$, then
 \begin{eqnarray} \label{ECT-3}
\lambda_{1}(B(q,r_0))\leq\lambda_{1}(\mathscr{B}_{n}(q^{-},r_0)),
 \end{eqnarray}
where $r_{0}<\min\{d_{\xi}(q),l\}$,
$M^{-}:=[0,l)\times_{f}\mathbb{S}^{n-1}$ is the spherically
symmetric $n$-manifold with the base point $q^{-}$ and with the
warping function $f$ determined by the following initial value
problem
\begin{eqnarray} \label{ODE}
\left\{
\begin{array}{lll}
f''(t)+\kappa(t)f(t)=0,  \qquad & 0<t<l,\\
f'(0)=1,~f(0)=0,\\
f(t)>0, \qquad & 0<t<l,
\end{array}
\right.
\end{eqnarray}
and $\mathscr{B}_{n}(q^{-},r_0)$ is the geodesic ball, with center
$q^{-}$ and radius $r_0$, on the model space $M^{-}$. Besides, the
equality in (\ref{ECT-3}) holds if and only if $B(q,r_0)$ is
isometric to $\mathscr{B}_{n}(q^{-},r_0)$.

\item If the radial sectional curvature of $M^{n}$ has an upper
bound\footnote{~In general, the bound function $\kappa(t)$ here is
different from the one used for the radial Ricci curvature, which
leads to the fact that generally the model spaces $M^{-}$ and
$M^{+}$ are different. However, for convenience and without
potential confusion, we wish to use $\kappa(t)$ all the time.}
$\kappa(t)$ with respect to a point $q\in M^{n}$, i.e.
 \begin{eqnarray*}
\mathscr{K}(V,\nu_{x})\leq\kappa(t)
 \end{eqnarray*}
for any $x=\gamma_{\xi}(t)=\exp_{q}(t\xi)$, $0<t<d_{\xi}(q)$,
$V\perp\nu_{x}$, $\|V\|=1$, then
 \begin{eqnarray} \label{ECT-4}
\lambda_{1}(B(q,r_0))\geq\lambda_{1}(\mathscr{B}_{n}(q^{+},r_0)),
 \end{eqnarray}
where $r_{0}<\min\{inj(q),l\}$,
$M^{+}:=[0,l)\times_{f}\mathbb{S}^{n-1}$ is the spherically
symmetric $n$-manifold with the base point $q^{+}$ and with the
warping function $f$ determined by (\ref{ODE}), and
$\mathscr{B}_{n}(q^{+},r_0)$ is the geodesic ball, with center
$q^{+}$ and radius $r_0$, on the model space $M^{+}$. Besides, the
equality in (\ref{ECT-3}) holds if and only if $B(q,r_0)$ is
isometric to $\mathscr{B}_{n}(q^{+},r_0)$.

\end{itemize}
It is not hard to see that if $\kappa(t)\equiv K$ is a constant
function, then eigenvalue comparisons (\ref{ECT-3}) and
(\ref{ECT-4}), together with related rigidity characterizations,
would degenerate into Cheng's conclusions
(\ref{ECT-1})-(\ref{ECT-2}) directly. By the way, in \cite[Section
2]{FMI} it has been shown that the optimal (lower or upper) bound of
the radial (Ricci or sectional) curvature with respect to a point on
a given complete $n$-manifold $M^n$ can always be found -- see
(2.9)-(2.10) therein for details. Specially, when $n=2$, i.e. $M^2$
is a complete surface, the radial Ricci curvature and the radial
sectional curvature coincide with each other, and they become
exactly the Gaussian curvature. In this situation, once a point
$q\in M^2$ is chosen, then the optimal lower and upper bounds of the
Gaussian curvature are actually the minimum value and the maximum
value of geodesic circles with the common center $q$. Furthermore,
if $M^2$ can be parameterized, then the optimal bounds for the
Gaussian curvature can be given in form and be computed numerically.
Consequently, by applying the eigenvalue comparisons (\ref{ECT-3})
and (\ref{ECT-4}), the optimal lower and upper bounds of
$\lambda_{1}(B(q,r_0))$ in this case can be computed numerically.
This is because, once the sharp curvature bound $\kappa(t)$ was
determined numerically, the warping function $f(t)$ can also be
determined numerically by solving the system (\ref{ODE}), and then
the optimal lower (or upper) bound
$\lambda=\lambda_{1}(\mathscr{B}_{2}(q^{+},r_0))$ (or
$\lambda=\lambda_{1}(\mathscr{B}_{2}(q^{-},r_0))$) of
$\lambda_{1}(B(q,r_0))$ can be calculated numerically by the system
\begin{eqnarray*}
\left\{
\begin{array}{lll}
\frac{d^{2}\varphi(t)}{dt^{2}}+\frac{f'(t)}{f(t)}\frac{d\varphi(t)}{dt}+\lambda\cdot\varphi(t)=0 \qquad &\mathrm{in}~\left(0,r_0\right),  \\[1mm]
\varphi'(0)=0,~\varphi\left(r_0\right)=0, \\[0.5mm]
 \varphi|_{(0,r_0)}>0.
\end{array} \right.
\end{eqnarray*}
Several interesting examples, involving torus, elliptic paraboloid
and saddle, have been shown in \cite[Section 6]{FMI} to intuitively
demonstrate the truth mentioned above. By the way, much more details
about the numerical calculations for those examples can be seen in
\cite[Section 2.5, Chapter 2]{JM3}.

A little bit later after the work \cite{FMI}, J. Mao \cite[Theorem
3.2]{JM4} successfully improved the eigenvalue comparison
(\ref{ECT-3}) to the case of nonlinear $p$-Laplacian. However, that
time he could not improve the eigenvalue comparison (\ref{ECT-4}) to
the case of
 $p$-Laplacian, since one of the main tools used in the proof of
 (\ref{ECT-4}) given in \cite{FMI}, i.e. Barta's lemma (see
 \cite{JB}), cannot be generalized to a suitable nonlinear version.
Very recently, he overcame this difficulty, that is, for complete
manifolds with \emph{radial} (Ricci or sectional) curvature bounded
(from below or from above), he not only gave Cheng-type eigenvalue
comparison theorems for the first Dirichlet eigenvalue of the
Witten-Laplacian (see \cite[Theorems 1.17 and 1.18]{JM-1}), which
are actually the weighted version of eigenvalue comparisons
(\ref{ECT-1})-(\ref{ECT-2}), but also gave Cheng-type eigenvalue
comparison theorems for the first Dirichlet eigenvalue of the
weighted $p$-Laplacian (see \cite[Theorems 6.6 and 7.3]{JM-1}).

\subsection{Our main conclusions}  \label{subsect3}

 The purpose of this
paper is trying to get Cheng-type eigenvalue comparisons (for the
first Dirichlet eigenvalues of the Witten-Laplacian and the weighted
$p$-Laplacian) under a weaker curvature assumption than the one made
by A. G. Setti in \cite{SAG}.

In fact, we can prove:

\begin{theorem} \label{maintheorem1}
Let $M^n$ be an $n$-dimensional ($n\geq2$) complete Riemannian
manifold with weight function\footnote{~We wish to make an agreement
that, for the purpose of convenience, in the sequel, if we say that
$w$ is a weight function on $M^n$, it implies that $w$ is a smooth
strictly positive function defined on $M^n$.} $w$. For a point $q\in
M^n$, if its modified Ricci curvature $R_{w}$ satisfies
\begin{eqnarray} \label{CS-1}
R_{w}\left(\frac{d}{dt}\Big{|}_{x},\frac{d}{dt}\Big{|}_{x}\right)\geq
n\kappa(t)=-n\frac{f''(t)}{f(t)},
\end{eqnarray}
with $x=\exp_{q}(t\xi)$ and $t<\min\{d_{\xi}(q),l\}$, then we have
\begin{eqnarray} \label{ECT-5}
\lambda_{1,w}(B(q,r_0))\leq\lambda_{1}(\mathscr{B}_{n+1}(q^{-},r_0)),
\end{eqnarray}
where $r_{0}<\min\{\ell(q),l\}$,
$\ell(q):=\sup_{M^n}r(x)=\max_{\xi}d_{\xi}$,
$M^{-}_{1}:=[0,l)\times_{f(t)}\mathbb{S}^n$ is the
$(n+1)$-dimensional spherically symmetric manifold with the base
point $q^{-}$, and moreover, $l$ and the warping function $f(t)$ are
determined by the initial value problem (\ref{ODE}). Besides,
similarly, $\mathscr{B}_{n+1}(q^{-},r_0)$ is the geodesic ball, with
center $q^{-}$ and radius $r_0$, on the model space $M^{-}_{1}$.
\end{theorem}

\begin{remark}
\rm{ In Theorem \ref{maintheorem1}, it is not necessary to require
that $B(q,r_0)$ is within the cut locus $\mathrm{Cut}(q)$ of $q$,
since the radius $r_0$ might choose value in the interval $
[inj(q),\ell(q))$, and in this situation, a part of the geodesic
ball $B(q,r_0)$ should be outside $\mathrm{Cut}(q)$. }
\end{remark}

Applying Theorem \ref{maintheorem1}, one directly has:

\begin{corollary} \label{coro-1}
 If $\kappa(t)$ in the curvature assumption (\ref{CS-1}) of Theorem
 \ref{maintheorem1} degenerates into $\kappa(t)\equiv K$ for some
 constant $K$, then
\begin{eqnarray} \label{ECT-6}
\lambda_{1,w}(B(q,r_0))\leq\lambda_{1}(\mathcal{B}_{n+1}(K,r_0)),
\end{eqnarray}
with $\mathcal{B}_{n+1}(K,r_0)$ a geodesic ball of radius $r_0$ in
the space $(n+1)$-form $\mathbb{M}^{n+1}(K)$ with constant sectional
curvature $K$.
\end{corollary}

\begin{remark}
\rm{ It is easy to see that Corollary \ref{coro-1} coincides exactly
with \cite[Theorem 4.2]{SAG}. That is to say, our Theorem
\ref{maintheorem1} here covers A. G. Setti's conclusion
\cite[Theorem 4.2]{SAG} as a special case. }
\end{remark}

\begin{theorem} \label{maintheorem2}
Under the assumptions as in Theorem  \ref{maintheorem1}, for the
weighted $p$-Laplacian ($1<p<\infty$), one has
\begin{eqnarray}  \label{ECT-7}
\lambda_{1,p}^{w}(B(q,r_0))\leq\lambda_{1,p}(\mathscr{B}_{n+1}(q^{-},r_0)),
\end{eqnarray}
 where  $\lambda_{1,p}(\cdot)$ denotes the
first Dirichlet eigenvalue of the $p$-Laplacian $\Delta_{p}$ on a
given geometric object.
\end{theorem}

Applying Theorem \ref{maintheorem2}, one directly has:

\begin{corollary} \label{coro-2}
If $\kappa(t)$ in the curvature assumption (\ref{CS-1}) of Theorem
 \ref{maintheorem1} degenerates into $\kappa(t)\equiv K$ for some
 constant $K$, then
  \begin{eqnarray} \label{ECT-8}
\lambda_{1,p}^{w}(B(q,r_0))\leq\lambda_{1,p}(\mathcal{B}_{n+1}(K,r_0)).
\end{eqnarray}
\end{corollary}

\begin{remark}
\rm{ It is not hard to see that if $p=2$, then operators
$\mathcal{L}_{p}$, $\Delta_{p}$ degenerate into $\mathcal{L}$ and
$\Delta$, respectively. Moreover, in this situation, the eigenvalue
comparison (\ref{ECT-8}) would become the eigenvalue comparison
(\ref{ECT-6}) exactly, which implies Corollary \ref{coro-2} covers
Corollary \ref{coro-1} as a special case. That is to say, comparing
with the assertion in Theorem
 \ref{maintheorem1}, here we also make an improvement to A. G. Setti's
conclusion \cite[Theorem 4.2]{SAG} in another direction. }
\end{remark}

\begin{remark}
\rm{ In March 2026, the corresponding author has already announced
the results of Theorems \ref{maintheorem1} and \ref{maintheorem2}
roughly in Remark 1.22 of his work \cite{JM-1}. However, until now,
we find enough time to write down all the details. }
\end{remark}

\section{A generalization of Bishop's volume comparison} \label{sect3}
\renewcommand{\thesection}{\arabic{section}}
\renewcommand{\theequation}{\thesection.\arabic{equation}}
\setcounter{equation}{0}

In this section, we wish to give a volume comparison under an
assumption concerning $R_{w}$, which can be seen as an extension of
the classical Bishop's volume comparison, and play an important role
in the proofs of Theorems \ref{maintheorem1} and \ref{maintheorem2}.

\begin{theorem}  \label{theorem3}
Let $M^n$ be an $n$-dimensional ($n\geq2$) complete Riemannian
manifold with weight function $w$. Under the curvature assumption as
in Theorem \ref{maintheorem1}, one has
 \begin{eqnarray*}
(w\sqrt{|g|})^{-1}\frac{\partial (w\sqrt{|g|})}{\partial t}\leq
n\frac{f'(t)}{f(t)},
 \end{eqnarray*}
with $t<\min\{d_{\xi}(q),l\}$, and $|g|:=\det[g_{ij}]$.
\end{theorem}

\begin{proof}
As in \cite[p. 703]{FMI}, one can define the path of linear
transformations
$\mathbb{A}(t,\xi):\xi^{\bot}\rightarrow{\xi^{\bot}}$ as follows
 \begin{eqnarray} \label{PLT}
\mathbb{A}(t,\xi)\eta=(\tau_{t})^{-1}Y_{\eta}(t),
 \end{eqnarray}
where
  $\xi^{\bot}$ is the orthogonal complement of
$\{\mathbb{R}\xi\}$ in $T_{q}M^{n}$,
$\tau_{t}:T_{q}M^{n}\rightarrow{T_{\exp_{q}(t\xi)}M^{n}}$ is the
parallel translation along $\gamma_{\xi}(t)$,
$Y_{\eta}(t)=d(\exp_q)_{(t\xi)}(t\eta)$ is the Jacobi field along
$\gamma_{\xi}(t)$ satisfying $Y_{\eta}(0)=0$, and
$(\nabla_{t}Y_{\eta})(0)=\eta$. This operator satisfies the Jacobi
equation
 $\mathbb{A}''+\mathcal{R}\mathbb{A}=0$ with initial conditions
 $\mathbb{A}(0,\xi)=0$, $\mathbb{A}'(0,\xi)=I$, where
$\mathcal{R}(t)$ is the self-adjoint operator on $\xi^{\bot}$, $
\mathcal{R}(t)\eta=(\tau_{t})^{-1}R(\gamma'_{\xi}(t),\tau_{t}\eta)
\gamma'_{\xi}(t)$. Then one easily has
\begin{eqnarray} \label{PT31-1}
\det \mathbb{A}(t,\sigma) = \sqrt{|g(t,\xi)|},
\end{eqnarray}
and the volume element can be written as
\begin{eqnarray*}
dv = \sqrt{|g(t,\xi)|}dtd\sigma,
\end{eqnarray*}
where $d\sigma$ denotes the $(n-1)$-dimensional volume element on
$\mathbb{S}^{n-1}\equiv S^{n-1}_{q}\subseteq T_{q}M^{n}$.

Define $U:= \mathbb{A}'(t,\xi)\mathbb{A}^{-1}(t,\xi)$ (unless
otherwise specified, $\mathbb{A}'$ denotes the derivative of
$\mathbb{A}$ with respect to the parameter $t$). By a direct
calculation, we have
\begin{eqnarray*}
\frac{\partial}{\partial t} \ln \det \mathbb{A}=\mathrm{tr}
(\mathbb{A}'\mathbb{A}^{-1})=\mathrm{tr} U.
\end{eqnarray*}
On the other hand, it follows from (\ref{PT31-1}) that
\begin{eqnarray*}
\frac{\partial}{\partial t}\ln \det
\mathbb{A}=\frac{\partial}{\partial
t}\ln\sqrt{|g|}=\frac{1}{\sqrt{|g|}}\frac{\partial
\sqrt{|g|}}{\partial t}.
\end{eqnarray*}
Hence, one has
\begin{eqnarray*}
\mathrm{tr} U=\frac{1}{\sqrt{|g|}}\frac{\partial
\sqrt{|g|}}{\partial t}.
\end{eqnarray*}

Let $\{\vec{e}_{i}\}_{i=1}^{n-1}$ be a parallel orthonormal frame
field along each geodesic $\gamma_{\xi}(t)$, $\xi\in S_{q}M^{n}$.
Then Jacobi field $Y_{\eta}(t)$ can be written as
\begin{eqnarray*}
Y_{\eta}(t) = \sum\limits_{i=1}^{n-1} A_i(t)\vec{e}_{i}(t),
\end{eqnarray*}
where $\mathbb{A}(t)=(A_{1}(t),A_{2}(t),\cdots,A_{n-1}(t))$, and
moreover, in this setting, the curvature matrix
$\mathcal{R}(t)=(R_{ij}(t))$ should satisfy
\begin{eqnarray*}
R_{ij}(t)=\left\langle R\left(\vec{e}_{i},\frac{\partial}{\partial
t}\right)\frac{\partial}{\partial t}, \vec{e}_{j} \right\rangle,
\end{eqnarray*}
since $\tau_{t}:T_{q}M^{n}\rightarrow{T_{\exp_{q}(t\xi)}M^{n}}$ is
the parallel translation along $\gamma_{\xi}(t)$ (i.e. isomorphism).

Since $U = \mathbb{A}' \mathbb{A}^{-1}$, a direct computation yields
\begin{eqnarray*}
U' =\mathbb{A}'' \mathbb{A}^{-1} - \mathbb{A}' \mathbb{A}^{-1}
\mathbb{A}' \mathbb{A}^{-1} = \mathbb{A}'' \mathbb{A}^{-1} - U^2.
\end{eqnarray*}
On the other hand, one has from the equation
$\mathbb{A}''+\mathcal{R}\mathbb{A}=0$ that
\begin{eqnarray*}
\mathbb{A}''\mathbb{A}^{-1} = -\mathcal{R}(t).
\end{eqnarray*}
So, one can get
\begin{eqnarray}  \label{PT31-2}
U' + U^2 + \mathcal{R}(t) = 0.
\end{eqnarray}
Since
 \begin{eqnarray*}
\mathrm{tr}\mathcal{R}(t)=\mathrm{tr}\left\langle
R\left(\vec{e}_{i},\frac{\partial}{\partial
t}\right)\frac{\partial}{\partial t}, \vec{e}_{j}
\right\rangle=\mathrm{Ric}\left(\frac{\partial}{\partial
t},\frac{\partial}{\partial t}\right),
 \end{eqnarray*}
it follows from (\ref{PT31-2}) that
 \begin{eqnarray} \label{PT31-3}
\mathrm{tr}U'+\mathrm{tr}U^{2}+\mathrm{Ric}\left(\frac{\partial}{\partial
t}, \frac{\partial}{\partial t}\right)=0.
 \end{eqnarray}

We now derive the following facts:

\vspace{2mm}

\noindent (I) $\mathrm{tr} U^{2}\geq\frac{(\mathrm{tr}
U)^{2}}{n-1}$,

\vspace{2mm}

\noindent (II) $\mathrm{Ric}(\frac{\partial}{\partial
t},\frac{\partial}{\partial t})\geq
-n\frac{f''}{f}+(w^{-1}\frac{\partial w}{\partial
t})^{2}+(w^{-1}\frac{\partial w}{\partial t})'$,

\vspace{2mm}

\noindent (III) $\frac{A^{2}}{n-1}+B^{2}\geq\frac{(A+B)^{2}}{n}$,
for all the real numbers $A, B$.

\vspace{2mm}

We first show (I). By (\ref{PLT}), it is not hard to see that the
operator $U=\mathbb{A}'\mathbb{A}^{-1}$ is self-adjoint. Let
$\Lambda_1, \Lambda_2, \ldots, \Lambda_{n-1}$ be the eigenvalues of
$U$. Then
\begin{eqnarray*}
\mathrm{tr}(U^2) = \sum_{i=1}^{n-1} \Lambda_i^2, \qquad (\mathrm{tr}
U)^2 = \left( \sum_{i=1}^{n-1} \Lambda_i \right)^2.
\end{eqnarray*}
By the Cauchy-Schwarz inequality, one has
\begin{eqnarray*}
\sum_{i=1}^{n-1} \Lambda_i^2 \geq \frac{\left( \sum_{i=1}^{n-1}
\Lambda_i \right)^2}{n-1},
\end{eqnarray*}
which implies (I) directly.

For (II), since the curvature assumption is
\begin{eqnarray*}
R_{w}\left(\frac{d}{dt}\Big{|}_{x},\frac{d}{dt}\Big{|}_{x}\right)\geq
n\kappa(t)=-n\frac{f''(t)}{f(t)},
\end{eqnarray*}
which implies
 \begin{eqnarray} \label{PT31-4}
\mathrm{Ric}\left(\frac{d}{dt}\Big{|}_{x},\frac{d}{dt}\Big{|}_{x}\right)=\mathrm{Ric}\left(\frac{\partial}{\partial
t}, \frac{\partial}{\partial t}\right) &\geq& -n \frac{f''}{f} +
w^{-1} \mathrm{Hess}(w)\left(\frac{\partial}{\partial t},
\frac{\partial}{\partial t}\right)\nonumber\\
&=& -n \frac{f''}{f} + w^{-1}\frac{\partial^{2}w}{\partial t}.
 \end{eqnarray}
 Set $h=\ln w$, and then
 \begin{eqnarray*}
\frac{\partial h}{\partial t}=w^{-1}\frac{\partial w}{\partial t},
\qquad  \frac{\partial ^{2}h}{\partial
t^{2}}=w^{-1}\frac{\partial^{2}w}{\partial
t^{2}}-\left(w^{-1}\frac{\partial w}{\partial t}\right)^{2},
\end{eqnarray*}
which implies
 \begin{eqnarray} \label{PT31-5}
w^{-1}\mathrm{Hess}(w)=w^{-1}\frac{\partial^{2}w}{\partial
t^{2}}=\left(w^{-1}\frac{\partial w}{\partial
t}\right)^{2}+\left(w^{-1}\frac{\partial w}{\partial t}\right)'.
 \end{eqnarray}
The fact (II) follows by  substituting (\ref{PT31-5}) into
(\ref{PT31-4}) directly.

In order to get (III),  it is sufficient to show
 \begin{eqnarray} \label{PT31-6}
\frac{A^{2}}{n-1}+B^{2}-\frac{(A+B)^{2}}{n}\geq 0.
\end{eqnarray}
Since
\begin{equation*}
\begin{aligned}
\frac{A^{2}}{n-1}+B^{2}-\frac{A^{2}+B^{2}+2AB}{n}
&= \frac{A^{2}}{n(n-1)}+\frac{(n-1)B^{2}}{n}-\frac{2AB}{n}  \\
&=\frac{1}{n(n-1)}(A^{2}+(n-1)^{2}B^{2}-2(n-1)AB)  \\
&=\frac{1}{n(n-1)}(A-(n-1)B)^{2}  \\
&\geq 0,
\end{aligned}
\end{equation*}
which directly implies (\ref{PT31-6}), the fact (III) follows
immediately.

By using the facts (I)-(III), together with the equation
(\ref{PT31-3}), one can obtain
\begin{equation*}
\begin{aligned}
0 &\geq \mathrm{tr} U'+\frac{(\mathrm{tr} U)^{2}}{n-1}+\left(w^{-1}\frac{\partial w}{\partial t}\right)^{2}+\left(w^{-1}\frac{\partial w}{\partial t}\right)'-n\frac{f''}{f} \nonumber\\
&=\left(\mathrm{tr} U+w^{-1}\frac{\partial w}{\partial t}\right)'+\frac{(\mathrm{tr} U)^{2}}{n-1}+\left(w^{-1}\frac{\partial w}{\partial t}\right)^{2}-n\frac{f''}{f} \nonumber \\
&\geq \left(\mathrm{tr} U+w^{-1}\frac{\partial w}{\partial t}\right)'+\frac{(\mathrm{tr} U+w^{-1}\frac{\partial w}{\partial t})^{2}}{n}-n\frac{f''}{f}. \\
\end{aligned}
\end{equation*}
Let $\varphi:=\mathrm{tr} U+w^{-1}\frac{\partial w}{\partial t}$. It
is easy to get
 \begin{eqnarray*}
\varphi=\frac{1}{\sqrt{|g|}}\frac{\partial \sqrt{|g|}}{\partial
t}+w^{-1}\frac{\partial w}{\partial
t}=(w\sqrt{|g|})^{-1}\frac{\partial (w\sqrt{|g|})}{\partial t}
 \end{eqnarray*}
by a direct computation. Therefore, one has
 \begin{eqnarray} \label{PT31-7}
\varphi'+\frac{\varphi^{2}}{n}-n\frac{f''}{f}\leq 0.
 \end{eqnarray}
Set $\psi (t)=n\frac{f'(t)}{f(t)}$. Since $f''(t)+\kappa(t)f(t)=0$,
one has
 \begin{eqnarray*}
\psi'=n\frac{f''f-(f')^{2}}{f^{2}}=n\frac{-\kappa
f^{2}-(f')^{2}}{f^{2}}=-n\kappa-n\left(\frac{f'}{f}\right)^{2},
 \end{eqnarray*}
that is,
 \begin{eqnarray}  \label{PT31-8}
\psi'+\frac{\psi^{2}}{n}+n\kappa(t)=\psi'+\frac{\psi^{2}}{n}-n\frac{f''}{f}=0.
 \end{eqnarray}

Assume that $\varphi (0)=\psi (0)$. Set
$\mathcal{W}(t):=\varphi-\psi(t)$, and then it follows from
(\ref{PT31-7})-(\ref{PT31-8}) that
 \begin{eqnarray*}
\mathcal{W}'(t)=\varphi'-\psi'\leq-\frac{\varphi^{2}-\psi^{2}}{n^2}=-\frac{\mathcal{W}(t)\cdot(\varphi+\psi)}{n^2},
 \end{eqnarray*}
which implies
 \begin{eqnarray} \label{PT31-9}
 \mathcal{W}'+\alpha(t)\mathcal{W}\leq0,
 \end{eqnarray}
 with $\alpha(t):=(\varphi+\psi)/n$. Let
 $\beta(t):=e^{\int_{0}^{t}\alpha(s)ds}$. By (\ref{PT31-9}), it is not
 hard to see
  \begin{eqnarray*}
 (\beta\mathcal{W})'=\beta\left(\mathcal{W}'+\alpha(t)\mathcal{W}\right)\leq0,
  \end{eqnarray*}
 and then together with $\varphi (0)=\psi (0)$, one has
 \begin{eqnarray*}
\beta(t)\mathcal{W}(t)\leq\beta(0)\mathcal{W}(0)=e^{0}\cdot\left(\varphi
(0)-\psi (0)\right)=0.
 \end{eqnarray*}
 Hence, it follows
 \begin{eqnarray*}
\mathcal{W}=\varphi-\psi(t)\leq0,
 \end{eqnarray*}
which implies
\begin{eqnarray*}
(w\sqrt{|g|})^{-1}\frac{\partial (w\sqrt{|g|})}{\partial t}\leq
n\frac{f'(t)}{f(t)}
\end{eqnarray*}
directly. This completes the proof.
\end{proof}

\section{The case of Witten-Laplacian} \label{sect4}
\renewcommand{\thesection}{\arabic{section}}
\renewcommand{\theequation}{\thesection.\arabic{equation}}
\setcounter{equation}{0}

Applying Theorem \ref{theorem3} as one of the main tools, we have:

\begin{proof} [Proof of Theorem \ref{maintheorem1}]
By using \cite[Lemma 3.2]{FMI} directly, one knows that:
\begin{itemize}
\item The eigenfunction corresponding to the first Dirichlet eigenvalue
of the Laplacian on $\mathscr{B}_{n+1}(q^{-},r_0)$ may be chosen to
be non-negative and is a radial function $T(t)$ satisfying $T'(t)<0$
for $0<t<r_0$.
\end{itemize}
Besides, in geodesic polar coordinates $(t,\xi)$ centered at
$q^{-}\in \mathscr{B}_{n+1}(q^{-},r_0)$, the Laplacian $\Delta$ on
$\mathscr{B}_{n+1}(q^{-},r_0)$ should take the form
\begin{eqnarray*}
\Delta=\frac{d^{2}}{dt^{2}}+n\frac{f'(t)}{f(t)}\frac{d}{dt}
+\frac{1}{f^{2}(t)}\Delta_{\mathbb{S}^{n}},
\end{eqnarray*}
where $\Delta_{\mathbb{S}^{n}}$ denotes the Laplacian on
$\mathbb{S}^n$ with respect to its round metric. Moreover,
 \begin{eqnarray*}
\Delta T(t)=T''(t) + n \frac{f'(t)}{f(t)} T'(t).
 \end{eqnarray*}
Therefore, combining the above facts, it is not hard to know that
$T(t)$ should satisfy
\begin{eqnarray} \label{firsteigenfunction-1}
\left\{
\begin{array}{lll}
T''(t) + n \frac{f'(t)}{f(t)} T'(t) + \lambda_{1}\left(\mathscr{B}_{n+1}(q^{-},r_0)\right)T(t) = 0 \qquad &\mathrm{in}~\left(0,r_0\right),  \\[1mm]
T'(0)=0,~T\left(r_0\right)=0, \\[0.5mm]
 T|_{(0,r_0)}>0,
\end{array} \right.
\end{eqnarray}
where $T\left(r_0\right)=0$ corresponds to the Dirichlet boundary
condition, and $T'(0)=0$ in (\ref{firsteigenfunction-1}) should be
required to assure the regularity of the eigenfunction $T(t)$ (see
e.g. \cite[Section 4, Chapter II]{IC} for a detailed explanation).

Define $\mathcal {G}(x):=T(r(x))$, where as in Subsection
\ref{subsect2}, $r(x)=d(q,x)$, i.e. the Riemannian distance between
$q$ and $x$. It is not hard to see that $\mathcal {G}= 0$ on
$\partial B(q, r_0)$, and $\mathcal {G}\in W^{1,2}_{0,w}(B(q,r_0))$.
Therefore, $\mathcal{G}$ can be used as a trial function for the
Dirichlet eigenvalue problem (\ref{eigenp-1}) of the operator
$\mathcal{L}$. By the variational characterization (\ref{chr-1}),
one easily has
 \begin{eqnarray} \label{PT4-1}
\lambda_{1,w}(B(q, r_0))\leq \frac{\int_{B(q, r_0)}|\nabla
\mathcal{G}|^{2}wdv}{\int_{B(q, r_0)}|\mathcal{G}|^{2}wdv}.
 \end{eqnarray}
Since $\mathcal{G}$ is radial, $|\nabla\mathcal{G}|^2=|T'|^2$, and
by a direct calculation under the geodesic polar coordinates around
$q$, we have
 \begin{eqnarray*}
\int_{B(q, r_0)}|\nabla
\mathcal{G}|^{2}wdv=\int_{\mathbb{S}^{n-1}}\int^{r_0}_{0}|T'|^{2}w\sqrt{|g|}dtd\sigma,
 \end{eqnarray*}
which, together with Theorem \ref{theorem3} and
(\ref{firsteigenfunction-1}), implies
\begin{equation*}
\begin{aligned}
\int^{r_0}_{0}|T'|^{2}w\sqrt{|g|}dt
&=TT'w\sqrt{|g|}\Big|^{r_0}_{0}-\int^{r_0}_{0}T(T'w\sqrt{|g|})'dt  \\
&=-\int^{r_0}_{0}T\left[T''w\sqrt{|g|}+T'(w\sqrt{|g|})'\right]dt  \\
&=-\int^{r_0}_{0}T\left[T''+T'\frac{(w\sqrt{|g|})'}{w\sqrt{|g|}}\right]w\sqrt{|g|}dt  \\
&\leq -\int^{r_0}_{0}T\left(T''+n\frac{f'}{f}T'\right)w\sqrt{|g|}dt  \\
&=\lambda_{1}\left(\mathscr{B}_{n+1}(q^{-},r_0)\right)\int^{r_0}_{0}T^{2}w\sqrt{|g|}dt.
\end{aligned}
\end{equation*}
So, we have
\begin{eqnarray*}
\int_{\mathbb{S}^{n-1}}\int^{r_0}_{0}|T'|^{2}w\sqrt{|g|}dtd\sigma\leq\lambda_{1}\left(\mathscr{B}_{n+1}(q^{-},r_0)\right)\int_{\mathbb{S}^{n-1}}\int^{r_0}_{0}T^{2}w\sqrt{|g|}dtd\sigma,
\end{eqnarray*}
which implies
\begin{eqnarray} \label{PT4-2}
\int_{B(q, r_0)}|\nabla
\mathcal{G}|^{2}wdv\leq\lambda_{1}\left(\mathscr{B}_{n+1}(q^{-},r_0)\right)\int_{B(q,
r_0)}|\mathcal{G}|^{2}wdv.
\end{eqnarray}
Combining (\ref{PT4-1}) and  (\ref{PT4-2}), one easily has
\begin{eqnarray*}
\lambda_{1,w}(B(q, r_0))\leq \frac{\int_{B(q, r_0)}|\nabla
\mathcal{G}|^{2}wdv}{\int_{B(q,
r_0)}|\mathcal{G}|^{2}wdv}\leq\lambda_{1}\left(\mathscr{B}_{n+1}(q^{-},r_0)\right).
\end{eqnarray*}
This completes the proof.
\end{proof}

\section{The case of weighted $p$-Laplacian}  \label{sect5}
\renewcommand{\thesection}{\arabic{section}}
\renewcommand{\theequation}{\thesection.\arabic{equation}}
\setcounter{equation}{0}

At the end, by applying Theorem \ref{theorem3} as one of the main
tools, we have:

\begin{proof} [Proof of Theorem \ref{maintheorem2}]
By using \cite[Proposition 3.1]{JM4} and the first part in the proof
of \cite[Theorem 3.2]{JM4}, one knows that:
\begin{itemize}
\item The eigenfunction corresponding to the first Dirichlet eigenvalue
of the $p$-Laplacian on $\mathscr{B}_{n+1}(q^{-},r_0)$ may be chosen
to be non-negative and is a radial function $\Psi(t)$ satisfying
$\Psi'(t)<0$ for $0<t<r_0$.
\end{itemize}
Besides, in geodesic polar coordinates $(t,\xi)$ centered at
$q^{-}\in \mathscr{B}_{n+1}(q^{-},r_0)$, the $p$-Laplacian
$\Delta_{p}$ on $\mathscr{B}_{n+1}(q^{-},r_0)$ should take the form
 \begin{eqnarray*}
\Delta_{p}=|\nabla
(\cdot)|^{p-2}\frac{d^{2}}{dt^{2}}+\frac{d}{dt}(|\nabla
(\cdot)|^{p-2})\frac{d}{dt}+n\frac{f'(t)}{f(t)}|\nabla
(\cdot)|^{p-2}\frac{d}{dt}+\frac{1}{f^{2}(t)}\Delta
_{p,\mathbb{S}^{n}},
 \end{eqnarray*}
where $\Delta_{p,\mathbb{S}^{n}}$ denotes the $p$-Laplacian on
$\mathbb{S}^n$ with respect to its round metric. Moreover,
 \begin{eqnarray*}
\Delta_{p}\Psi(t)=|\Psi'(t)|^{p-2}\Psi''(t)+\frac{d}{dt}(|\Psi'(t)|^{p-2})\Psi'(t)+n\frac{f'(t)}{f(t)}|\Psi'(t)|^{p-2}\Psi'(t).
\end{eqnarray*}
Therefore, it is not hard to know from the above facts that
$\Psi(t)$ should satisfy
\begin{eqnarray} \label{firsteigenfunction-2}
\left\{
\begin{array}{llll}
(p-1)|\Psi'(t)|^{p-2}\Psi''(t)+n\frac{f'(t)}{f(t)}|\Psi'(t)|^{p-2}\Psi'(t) + \\
\qquad\qquad\qquad\qquad\qquad \lambda_{1}\left(\mathscr{B}_{n+1}(q^{-},r_0)\right)|\Psi(t)|^{p-2}\Psi(t) = 0 \qquad &\mathrm{in}~\left(0,r_0\right),  \\[1mm]
\Psi'(0)=0,~\Psi\left(r_0\right)=0, \\[0.5mm]
 \Psi|_{(0,r_0)}>0.
\end{array} \right.
\end{eqnarray}

Define $\mathcal {F}(x):=\Psi(r(x))$, where  $r(x)=d(q,x)$. It is
not hard to see that $\mathcal {F}= 0$ on $\partial B(q, r_0)$, and
$\mathcal {F}\in W^{1,p}_{0,w}(B(q,r_0))$. Therefore, $\mathcal{F}$
can be used as a trial function for the Dirichlet eigenvalue problem
(\ref{eigenp-wpl}) of the operator $\mathcal{L}_{p}$. By the
variational characterization (\ref{chr-2}), one easily has
 \begin{eqnarray} \label{PT5-1}
\lambda_{1,p}^{w}(B(q, r_0))\leq \frac{\int_{B(q, r_0)}|\nabla
\mathcal{F}|^{p}wdv}{\int_{B(q, r_0)}|\mathcal{F}|^{p}wdv}.
 \end{eqnarray}
Since $\mathcal{F}$ is radial, $|\nabla\mathcal{F}|^p=|\Psi'|^p$,
and by a direct calculation under the geodesic polar coordinates
around $q$, we have
 \begin{eqnarray*}
\int_{B(q, r_0)}|\nabla
\mathcal{F}|^{p}wdv=\int_{\mathbb{S}^{n-1}}\int^{r_0}_{0}|\Psi'|^{p}w\sqrt{|g|}dtd\sigma,
 \end{eqnarray*}
which, together with Theorem \ref{theorem3} and
(\ref{firsteigenfunction-2}), implies
\begin{equation*}
\begin{aligned}
\int^{r_0}_{0}|\Psi'|^{p}w\sqrt{|g|}dt
&=-\Psi|\Psi'|^{p-1}w\sqrt{|g|}\Big|^{R}_{0}-\int^{r_0}_{0}(-\Psi)(|\Psi'|^{p-1}w\sqrt{|g|})'dt  \\
&=\int^{r_0}_{0}\Psi\left[(p-1)|\Psi'|^{p-3}\Psi'\Psi''w\sqrt{|g|}+|\Psi'|^{p-1}(w\sqrt{|g|})'\right]dt  \\
&=\int^{r_0}_{0}\Psi\left[(p-1)|\Psi'|^{p-3}\Psi'\Psi''+|\Psi'|^{p-1}\frac{(w\sqrt{|g|})'}{w\sqrt{|g|}}\right]w\sqrt{|g|}dt   \\
&\leq \int^{r_0}_{0}\Psi\left[(p-1)|\Psi'|^{p-3}\Psi'\Psi''+n\frac{f'}{f}|\Psi'|^{p-1}\right]w\sqrt{|g|}dt  \\
&=\int^{r_0}_{0}\Psi\left[-(p-1)|\Psi'|^{p-2}\Psi''-n\frac{f'}{f}|\Psi'|^{p-2}\Psi'\right]w\sqrt{|g|}dt  \\
&=\lambda_{1,p}\left(\mathscr{B}_{n+1}(q^{-},r_0)\right)\int^{r_0}_{0}|\Psi|^{p}w\sqrt{|g|}dt.
\end{aligned}
\end{equation*}
So, we have
\begin{eqnarray*}
\int_{\mathbb{S}^{n-1}}\int^{r_0}_{0}|\Psi'|^{p}w\sqrt{|g|}dtd\sigma\leq\lambda_{1,p}\left(\mathscr{B}_{n+1}(q^{-},r_0)\right)\int_{\mathbb{S}^{n-1}}\int^{r_0}_{0}|\Psi|^{p}w\sqrt{|g|}dtd\sigma,
\end{eqnarray*}
which implies
\begin{eqnarray} \label{PT5-2}
\int_{B(q, r_0)}|\nabla
\mathcal{F}|^{p}wdv\leq\lambda_{1,p}\left(\mathscr{B}_{n+1}(q^{-},r_0)\right)\int_{B(q,
r_0)}|\mathcal{F}|^{p}wdv.
\end{eqnarray}
Combining (\ref{PT5-1}) and  (\ref{PT5-2}), one easily has
\begin{eqnarray*}
\lambda_{1,p}^{w}(B(q, r_0))\leq \frac{\int_{B(q, r_0)}|\nabla
\mathcal{F}|^{p}wdv}{\int_{B(q,
r_0)}|\mathcal{F}|^{p}wdv}\leq\lambda_{1,p}\left(\mathscr{B}_{n+1}(q^{-},r_0)\right).
\end{eqnarray*}
This completes the proof.
\end{proof}

\begin{remark}
\rm{ We wish to emphasize that for the curvature assumption in
Theorem \ref{maintheorem1} (or Theorem \ref{maintheorem2}), we do
not need to require that the weight function $w$ is radial with
respect to the chosen point $q$. This is the essential difference
between this paper and the recent work \cite{JM-1}, where Cheng-type
eigenvalue comparison theorems, including rigidity
characterizations, have been set up for the Witten-Laplacian and the
weighted $p$-Laplacian on complete manifolds with radial (Ricci or
sectional) curvature bounded, provided the weight function therein
is \emph{radial}. }
\end{remark}

\section*{Acknowledgments}
\renewcommand{\thesection}{\arabic{section}}
\renewcommand{\theequation}{\thesection.\arabic{equation}}
\setcounter{equation}{0} \setcounter{maintheorem}{0}

This research was supported in part by the NSF of China (Grant Nos.
11801496 and 11926352), the Fok Ying-Tung Education Foundation
(China), the Key Project of Jiangxi Provincial Natural Science
Foundation (Grant No. 20252BAC250003), Hubei Key Laboratory of
Applied Mathematics (Hubei University), and Key Laboratory of
Intelligent Sensing System and Security (Hubei University), Ministry
of Education.


\begin{thebibliography}{9999}


\bibitem{BE} D. Bakry, M. \'{E}mery, \emph{Diffusion hypercontractives}, S\'{e}m.
Prob. XIX. Lect. Notes Math. {\bf 1123} (1985) 177--206.

\bibitem{JB} J. Barta, \emph{Sur la vibration fundamentale d'une
membrane}, C. R. Acad. Sci. {\bf 204} (1937) 472--473.


\bibitem{IC} I. Chavel, \emph{Eigenvalues in Riemannian Geometry}, Academic Press, New
York (1984).


\bibitem{CM} R. F. Chen, J. Mao, \emph{Several isoperimetric inequalities of Dirichlet and Neumann eigenvalues of the Witten
Laplacian}, J. Spectral Theory {\bf 15} (2025) 1241--1277.


\bibitem{CM1} R. F. Chen, J. Mao, \emph{On the Ashbaugh-Benguria type conjecture about lower-order Neumann
eigenvalues of the Witten-Laplacian}, available online at
arXiv:2403.08070v3

\bibitem{CSY1} S. Y. Cheng, \emph{Eigenvalue comparison theorems and its geometric
applications}, Math. Zeit. {\bf143} (1975) 289--297.

\bibitem{CSY2} S. Y. Cheng, \emph{Eigenfunctions and eigenvalues of Laplacian},
Amer. Math. Soc. Proc. Symp. Pure Math. {\bf27} (Part II) (1975)
185--193.


\bibitem{DDMZ} Y. L. Deng, F. Du, J. Mao, Y. Zhao, \emph{Sharp eigenvalue estimates and related rigidity
theorems}, Hokkaido Math. J. {\bf 55} (2026) 57--85.


\bibitem{DMWW} F. Du, J. Mao, Q. L. Wang, C. X. Wu, \emph{Eigenvalue inequalities for the buckling problem of the
drifting Laplacian on Ricci solitons}, J. Differ. Equat. {\bf 260}
(2016) 5533--5564.

\bibitem{DMWX} F. Du, J. Mao, Q. L. Wang, C. Y. Xia, \emph{Estimates for eigenvalues of weighted Laplacian and
weighted $p$-Laplacian}, Hiroshima Math. J. {\bf 51} (2021)
335--353.

\bibitem{FMI} P. Freitas, J. Mao, I. Salavessa, \emph{Spherical symmetrization and the first eigenvalue of
geodesic disks on manifolds}, Calc. Var. Partial Differential
Equations {\bf 51} (2014) 701--724.


\bibitem{JM3} J. Mao, \emph{Eigenvalue estimation and some results on finite topological
type}, Ph.D. thesis, IST-UTL (2013).



\bibitem{JM4} J. Mao, \emph{Eigenvalue inequalities for the $p$-Laplacian on a Riemannian
manifold and estimates for the heat kernel}, J. Math. Pures Appl.
{\bf101} (2014) 372--393.


\bibitem{JM1} J. Mao, \emph{The Gagliardo-Nirenberg inequalities and manifolds with non-negative weighted Ricci
curvature}, Kyushu J. Math.  {\bf 70} (2016) 29--46.


\bibitem{JM2} J. Mao, \emph{Functional inequalities and manifolds with nonnegative weighted Ricci
curvature}, Czech. Math. J. {\bf 70} (2020) 213--233.



\bibitem{JM-1} J. Mao, \emph{Weighted heat kernel comparison theorems and its
applications in spectral geometry}, submitted and available online
at arXiv:2603.00942v3

\bibitem{SAG} A. G. Setti, \emph{Eigenvalue estimates for the weighted Laplacian on a
Riemannian manifold}, Rend. Sem. Mat. Univ. Padova {\bf 100} (1998)
27--55.


\bibitem{YWMD} Y. Zhao, C. X. Wu, J. Mao, F. Du, \emph{Eigenvalue comparisons in Steklov eigenvalue problem
and some other eigenvalue estimates}, Revista Matem\'{a}tica
Complutense {\bf 33} (2020) 389--414.






\end{thebibliography}
\end{document}